\documentclass[11pt]{article}
\usepackage[numbers,sort&compress]{natbib}
\usepackage{enumerate}
\usepackage{amscd}
\usepackage{amsmath}
\usepackage{latexsym}
\usepackage{amsfonts}
\usepackage{amssymb}
\usepackage{amsthm}
\usepackage{verbatim}
\usepackage{mathrsfs}
\usepackage{enumerate}
\usepackage{hyperref}

\theoremstyle{plain}\newtheorem{definition}{Definition}[section]
\theoremstyle{definition}\newtheorem{theorem}{Theorem}[section]
\theoremstyle{plain}\newtheorem{lemma}[theorem]{Lemma}
\theoremstyle{plain}\newtheorem{coro}[theorem]{Corollary}
\theoremstyle{plain}
\theoremstyle{remark}\newtheorem{remark}{Remark}[section]
\usepackage{xcolor}
\newcommand{\wblue}[1]{\textcolor{black}{#1}}
\newcommand{\wred}[1]{\textcolor{black}{#1}}

\newcommand{\Div}{\mathrm{div}\,}
\newcommand{\B}{\Big}

\newcommand{\R}{\mathbb{R}}
\newcommand{\be}{\begin{equation}}
\newcommand{\ee}{\end{equation}}
 \newcommand{\ba}{\begin{aligned}}
 \newcommand{\ea}{\end{aligned}}

\providecommand{\bysame}{\leavevmode\hbox to3em{\hrulefill}\thinspace}
  \newcommand{\f}{\frac}
    
  \newcommand{\ben}{\begin{enumerate}}
   \newcommand{\een}{\end{enumerate}}

\newcommand{\ti}{\nabla}

\newcommand{\Rmnum}[1]{\expandafter\@slowromancap\romannumeral #1@}

\allowdisplaybreaks

\numberwithin{equation}{section}
\begin{document}
\title{\wblue{Remarks on the singular set of suitable weak solutions to the 3D  Navier-Stokes equations}}
\author{Wei Ren\footnote{College of Mathematics and Systems Science, Hebei University, Baoding, Hebei 071000, P. R. China Email:  renwei4321@163.com},~~\;Yanqing Wang\footnote{ Department of Mathematics and Information Science, Zhengzhou University of Light Industry, Zhengzhou, Henan  450002,  P. R. China Email: wangyanqing20056@gmail.com}\; ~and Gang Wu\footnote{School of Mathematical Sciences,  University of Chinese Academy of Sciences, Beijing 100049, P. R. China Email: wugangmaths@gmail.com}}
\date{}
\maketitle\vspace{-0.85cm}
\begin{abstract}In this paper, let
$\mathcal{S}$ denote  the  possible   interior singular set of suitable weak solutions of the 3D Navier-Stokes equations. We improve the known   upper  box-counting  dimension of this set from  $360/277(\approx1.300)$ in \cite{[WW2]}
to  $975/758(\approx1.286)$.
It is also shown that $\Lambda(\mathcal{S},r(\log(e/r))^{\sigma})=0(0\leq\sigma<27/113)$, which extends the previous   corresponding  results concerning the  improvement  of  the classical Caffarelli-Kohn-Nirenberg theorem by a logarithmic factor in
Choe and    Lewis   \cite[ J. Funct. Anal., 175: 348-369, 2000]{[CL]} and in
Choe and  Yang et al. \cite[ Comm. Math. Phys, 336: 171-198, 2015]{[CY1]}.
 The proof  is inspired by  a new  $\varepsilon$-regularity criterion  proved by Guevara   and Phuc in \cite[Calc. Var.  56:68, 2017]{[GP]}.
 \end{abstract}
\section{Introduction}
\label{intro}
\setcounter{section}{1}\setcounter{equation}{0}
We consider  the following   incompressible Navier-Stokes equations in   three-dimensional space
\be\left\{\ba\label{NS}
&u_{t} -\Delta  u+ u\cdot\ti
u+\nabla \Pi=0, ~~\Div u=0,\\
&u|_{t=0}=u_0,
\ea\right.\ee
 where $u $ stands for the flow  velocity field, the scalar function $\Pi$ represents the   pressure.   The
initial  velocity $u_0$ satisfies   $\text{div}\,u_0=0$.

In a series of papers,  Scheffer \wblue{in} \cite{[Scheffer1],[Scheffer2],[Scheffer3]} proposed a program
to estimate the size of the potential space-time singular   set $\mathcal{S}$ of (suitable) weak solutions obeying the local energy inequality to the   Navier-Stokes system
  and proved that the Hausdorff dimension of this set \wblue{of the 3D Navier-Stokes equations  is at most $5/3$.  A point is said to be a regular point of the suitable weak solution} $u$ provided one
has the $L^{\infty} $ bound of $u$ in some neighborhood
of this point.    The remaining points are called  singular points. In this direction,
the celebrated Caffarelli-Kohn-Nirenberg theorem \wblue{in \cite{[CKN]} about} the 3D Navier-Stokes system is that   one
dimensional Hausdorff measure of $\mathcal{S}$ is zero, which is   deduced from    the following $\varepsilon$-regularity criterion:
 there is an absolute constant \wblue{$\varepsilon$} such that, if
   \be\label{ckn}
  \limsup_{r\rightarrow0} {r^{-\f12}}\|\nabla u  \|_{L_{t}^{2}L_{x}^{2}(Q(r))}
  \leq \varepsilon,
  \ee
  then $(x,\,t)$ is a regular point, where
   $Q(r):=B(r)\times(t-r^{2},t)$ and $B(r)$
   denotes the ball of center $x$ and radius $r$.
   From that time on, much effort has been devoted to
\wblue{the} extension of the Caffarelli-Kohn-Nirenberg theorem
 and the $\varepsilon$-regularity criteria were presented in    several works(see, e.g., \cite{[Choe],[CY2],[GP],[KY6],[Lin],[LS],[CL],
[CY1],[Kukavica],[KP],[Vasseur],[RS2],[WW1],[MW],[RS3],[Struwe],[WW2]}).

Recently, in view of Bernoulli (total) pressure \wblue{$\f{1}{2}|u|^{2}
+\Pi$} as a signed distribution belonging to certain fractional
Sobolev space of negative order in local energy inequality,
Guevara   and Phuc \wblue{in \cite{[GP]}   proved the following $\varepsilon$-regularity criterion: if
\be\label{GP}
\mu^{-\f{3}{2}}\Big(\Big\||u|^{2}\Big\|_{L_{t}^{p}L_{x}^{q}(Q(\mu))}
+\|\Pi\|_{L_{t}^{p}L_{x}^{q}(Q(\mu))}\Big)<\varepsilon_{0},
\ee}
where $(p,\,q)$ satisfying
\be\label{bpc1}{2}/{p}+{3}/{q}=
7/2~\text{with}~1\leq p\leq2,\ee
then $(x,t)$ is a regular point.
An especially interesting case of \eqref{GP} is $p=q=10/7$, which improves the \wblue{following classical one shown in \cite{[Lin],[LS]} via  blow-up procedure
\be\label{Lin}
\mu^{-\f{2}{3}}\Big(\Big\||u|^{2}\Big\|_{L_{t}^{3/2}L_{x}^{3/2}(Q(\mu))}
+\|\Pi\|_{L_{t}^{3/2}L_{x}^{3/2}(Q(\mu))}\Big)<\varepsilon.
\ee}
For the pair $(p,\,q)$  meeting with $\eqref{bpc1}$,
we would like to mention an $\varepsilon$-regularity criterion in terms of Bernoulli   pressure obtained in \cite{[MW]}
$$
\limsup_{\mu\rightarrow0}\mu^{-\f{3}{2}}
\Big(\B\|\f{1}{2}|u|^{2}
+\Pi\B\|_{L_{t}^{p}L_{x}^{q}(Q(\mu))}\Big)<\varepsilon.
 $$
One objective of this paper is to give an  improvement of the known fractal upper box
 dimension   of $\mathcal{S}$ via \eqref{GP}. The relationship between Hausdorff dimension and the upper box dimension  is  that the first one  is less than second one (see e.g. \cite{[Falconer]}). The definition of box dimension is via lower box dimension and upper box
dimension. In what follows, box dimension and fractal dimension mean  the upper box dimension.
 Before we \wblue{state} our theorem, we  recall
previous related results. With the help of \eqref{Lin},
Robinson  and   Sadowski \cite{[RS2]}
proved that the   upper box
dimension   of $\mathcal{S}$   is at most 5/3. Shortly afterwards,
Kukavica \cite{[Kukavica]} \wblue{showed}
  that the  box dimension of the
singular   set is less than or equal to $135/82(\approx1.646)$ and \wblue{proposed} a   question whether this dimension of the singular set is at most 1. It was shown that the parabolic fractal dimension
of the singular   set is less than or equal to $45/29(\approx1.552)$
 by Kukavica and Pei in \cite{[KP]}.
Very recently,   Koh and  Yang  \cite{[KY6]} proved that
the  fractal upper box
dimension  of $\mathcal{S}$    is bounded by $95/63(\approx1.508)$.
 In light of the arguments in \cite{[KY6]} and
 some delicate estimates, the authors \wblue{in \cite{[WW2]}} refined the
 upper box dimension to $360/277(\approx 1.300).$

\wblue{Our first result in this paper is the following theorem:}
\begin{theorem}\label{the1.1}
 The (upper) box dimension of  $\mathcal{S}$ is at most $975/758(\approx1.286).$
\end{theorem}
\begin{remark}
This improves the previous box dimension of   $\mathcal{S}$ obtained in \cite{[KP],[RS2],[Kukavica],[KY6],[WW2]}.
\end{remark}
By contradiction \wblue{arguments as in \cite{[WW2],[KP]}, Theorem  \ref{the1.1} turns out to be the  consequence of the following theorem.}
 \begin{theorem}\label{the1.2}
Suppose that the pair $(u, \,\Pi)$ is a suitable weak solution to (\ref{NS}).  Then, for any $\gamma <865/2274$, $(x ,\,t )$ is a regular point  provided there exist a sufficiently small universal positive constant $\varepsilon_{1}$ and $0<r<1$ such that
\be\label{cond}\ba   \iint_{ Q (r)}
|\nabla u |^{2} +| u |^{ 10/3}+|\Pi-\overline{\Pi}_{ B (r)} |^{ 5/3}+
 |\nabla \Pi| ^{5/4}dxds \leq    r^{5/3-\gamma}\varepsilon_{1}.  \ea\ee
\end{theorem}
 The notations used here can be found at the end of this section.
 \begin{remark}
Theorem \ref{the1.2} is an improvement of corresponding results
proved    in \cite{[KP],[WW2]}.
\end{remark}
\begin{remark}Theorem  \ref{the1.2} has been inspired by the new \wblue{$\varepsilon$-regularity criterion}  \eqref{GP}.
The main method in proving the above result is the one utilized in \cite{[KY6]}. Furthermore, motivated by \cite{[WW2]},  we  utilize
the quantities $\|\nabla\Pi\|^{5/4}_{L^{5/4}_{t,x}}$, $\|\nabla u\|^{2}_{L^{2}_{t,x}}$ bounded by the
initial energy as widely as possible. To apply \eqref{GP}, we need establish some decay estimates  	adapted to it, see Lemmas \ref{lemma2.1} and \ref{presure}, which    play
an important role in the proof.
\end{remark}
It is   known that the   Hausdorff dimension of the possible singular set
 of the suitable weak solution of 5D stationary  Navier-Stokes equations is also \wblue{at most 1} (see eg. \cite{[Struwe],[WW1]} and references therein). Therefore, a natural question is whether the box dimension of
the singular   set to the 5D stationary Navier-Stokes equations is at most one. Indeed, following the path of \cite{[KY6],[WW2]}, one could prove that
the (upper) box dimension of the set of possible  singular  sets  of suitable weak solutions to this system  is at most $15/13(\approx1.154).$ To this end,
one just  utilizes     an analogue of $\varepsilon$-regularity \wblue{criterion} \eqref{Lin}, since the $\varepsilon$-regularity \wblue{criterion} \eqref{GP} for   	time-independent
 equations yields the same result. We leave this  for the interested readers.

The celebrated
Caffarelli-Kohn-Nirenberg theorem for the three-dimensional time-dependent Navier-Stokes system
 can be written as $\Lambda(\mathcal{S},r)=0$, for the details of notation, see \wblue{Sections 2}.  Some authors improve the
 Caffarelli-Kohn-Nirenberg theorem by a logarithmic factor, see, for example, \cite{[CL],[CY1],[Choe],[CY2]}. Particularly,
 in \cite{[CL]}, Choe and Lewis introduced the generalized Hausdorff measure  $\Lambda(\mathcal{S},r(\log(e/r))^{\sigma})$ and proved that $\Lambda(\mathcal{S},r(\log(e/r))^{\sigma})=0  (0\leq\sigma<3/44)$.
By deriving a new local energy inequality in the  absence  of pressure, Choe and Yang \wblue{in} \cite{[CY1]} studied the regularity of
suitable weak solutions of the
magnetohydrodynamic equations in
dimension three and proved that $\Lambda(\mathcal{S},r(\log(e/r))^{\sigma})=0$, where $\mathcal{S}$ denotes the potential  interior singular set of suitable weak solutions for this system and  $\sigma$ is bounded by $1/6$.
The reader is referred to
the recent work \cite{[Choe],[CY2]}
for the boundary case.
The second goal of this paper is to improve  the bound of $\sigma$ mentioned above. Precisely, we have the following fact.
\begin{theorem}\label{main}
 Let $\mathcal{S}  $ stand   for  the set of all the potential interior singular set of suitable weak solutions  to \eqref{NS} and $0\leq\sigma<27/113$. There holds
$$\Lambda(\mathcal{S},r(\log(e/r))^\sigma)=0.$$
\end{theorem}
\begin{remark}
Theorem \ref{main} is an improvement of the known  corresponding results in  \cite{[CL],[CY1]}.
\end{remark}
\begin{remark}
A combination of arguments presents here and the $\varepsilon$-regularity criterion \eqref{Lin}  implies that, $\Lambda(\mathcal{S},r(\log(e/r))^\sigma)=0$ ($0\leq\sigma<5/22$). The $\varepsilon$-regularity criterion
 \eqref{GP}   combined with the proof of
Theorem \ref{main}   yields that  $\Lambda(\mathcal{S},r(\log(e/r))^\sigma)=0$ $(0\leq\sigma<28/127)$. Naturally, it may be helpful to
utilize the one with $p=10/3$ below
 \be\mu^{-(5-2p)}\Big(\Big\||u|^{2}\Big\|_{L_{t}^{p/2}L_{x}^{p/2}(Q(\mu))}
+\|\Pi\|_{L_{t}^{p/3}L_{x}^{p/2}(Q(\mu))}\Big)<\varepsilon. \label{wwww}\ee
However, one needs   $J_{q} (\rho)$ with $q=2$ in the proof, which contradicts   \eqref{first-integration}. Based on this, for any $\kappa>0$, we will apply  \eqref{wwww} with $p=10/3-\kappa$.
 This allows us to obtain the desired result.
\end{remark}

The remainder of this paper is divided into  three sections.
In Section  2, we present the definitions  of upper box-counting dimension and generalized Hausdorff measure. Then, we   recall the definition of suitable weak solutions to the Navier-Stokes equations and list some crucial  bounds for the scaling   invariant quantities.
The third section is devoted to the box-counting dimension
of the possible  singular set of suitable weak
solutions.
Section 4  is concerned with  generalized Hausdorff dimension of the potential singular   set  in the Navier-Stokes system.

\noindent
{\bf Notations:} Throughout this paper, the classical Sobolev norm $\|\cdot\|_{H^{s}}$  is defined as   $\|f\|^{2} _{{H}^{s}}= \int_{\mathbb{R}^{n}} (1+|\xi|)^{2s}|\hat{f}(\xi)|^{2}d\xi$, $s\in \mathbb{R}$.
  We denote by  $ \dot{H}^{s}$ homogenous Sobolev spaces with the norm $\|f\|^{2} _{\dot{H}^{s}}= \int_{\mathbb{R}^{n}} |\xi|^{2s}|\hat{f}(\xi)|^{2}d\xi$.
  Denote by $L_{\sigma}^{q}(\Omega)$ the closure of $C^{\infty}_{0,\sigma}(\Omega)$ in \wblue{$L^{q}(\Omega)^{n}$, where
 $C^{\infty}_{0,\sigma}(\Omega)=\{u\in C^{\infty}_{0}(\Omega)^{n};\, \text{div}\,u=0\}$.}
 The classical Sobolev space $W^{1,2}(\Omega)$ is equipped with the norm $\|f\|_{W^{1,2}(\Omega)}= \|f\|_{L^{2}(\Omega)}+ \|\nabla f\|_{L^{2}(\Omega)}$.
 For $q\in [1,\,\infty]$, the notation $L^{q}(0,\,T;\,X)$ stands for the set of measurable functions on the interval $(0,\,T)$ with values in $X$ and $\|f(t,\cdot)\|_{X}$ belongs to $L^{q}(0,\,T)$.
   For simplicity,   we write
$$\|f\| _{L^{q,\,\ell}(Q(\mu))}:=\|f\| _{L^{q}(t-\mu^{2},\,t;\,L^{\ell}(B(\mu)))}~~~\text{ and}~~~~
  \|f\| _{L^{q}(Q(\mu))}:=\|f\| _{L^{q,\,q}(Q(\mu))}.$$

Denote
  the average of $f$ on the set $\Omega$ by
  $\overline{f}_{\Omega}$. For convenience,
  \wblue{$\overline{f}_{r}$ represents  $\overline{f}_{B(r)}$.
$K$  stands for the standard normalized fundamental solution of Laplace equation in $\mathbb{R}^{n}$ with $n\geq2$.}
  $|\Omega|$ represents the Lebesgue measure of the set $\Omega$. We will use the summation convention on repeated indices.
 $C$ is an absolute constant which may be different from line to line unless otherwise stated in this paper.


  \section{Preliminaries}
  First, we begin with the \wblue{definitions} of the (upper) box-counting dimension of a set and  the generalized Hausdorff measure below, respectively.
\begin{definition}\label{defibox}
The (upper) box-counting dimension of a set $X$ is usually defined as
$$d_{\text{box}}(X)=\limsup_{\epsilon\rightarrow0}\f{\log N(X,\,\epsilon)}{-\log\epsilon},$$
where $N(X,\,\epsilon)$ is the minimum number of balls of radius $\epsilon$ required to cover $X$.
\end{definition}
\begin{definition}[cf. \cite{[CL]}]
Let $h$ be an increasing continuous function on $(0, 1]$ with $\lim\limits_{r\rightarrow 0} h(r)=0$ and $h(1)=1$.
For fixed parameter $\delta>0$ and set $E\subset \R^3 \times \R$, we denote by $D(\delta)$
the family of all coverings $\{Q(x_i,t_i;\,r_{i})\}$ of $E$ with $0<r_{i}\leq \delta$. We denote
$$ \Psi_\delta (E, h)=\inf_{D(\delta)}\sum_i h(r_i)$$
and define the generalized parabolic Hausdorff measure as
$$\Lambda(E,h)=\lim_{\delta\rightarrow 0}\Psi_\delta (E, h).$$
\end{definition}
Second, we recall the definition of suitable weak solutions to the Navier-Stokes equations \eqref{NS}.
\begin{definition}\label{defi}
A  pair   $(u, \,\Pi)$  is called a suitable weak solution to the Navier-Stokes equations \eqref{NS} provided the following conditions are satisfied,
\begin{enumerate}[(1)]
\item $u \in L^{\infty}(-T,\,0;\,L^{2}(\mathbb{R}^{3}))\cap L^{2}(-T,\,0;\,\dot{H}^{1}(\mathbb{R}^{3})),\,\Pi\in
L^{3/2}(-T,\,0;L^{3/2}(\mathbb{R}^{3}));$\label{SWS1}
 \item$(u, ~\Pi)$~solves (\ref{NS}) in $\mathbb{R}^{3}\times (-T,\,0) $ in the sense of distributions;\label{SWS2}
 \item$(u, ~\Pi)$ satisfies the following inequality, for a.e. $t\in[-T,0]$,
 \begin{align}
 &\int_{\mathbb{R}^{3}} |u(x,t)|^{2} \phi(x,t) dx
 +2\int^{t}_{-T}\int_{\mathbb{R} ^{3 }}
  |\nabla u|^{2}\phi  dxds\nonumber\\ \leq&  \int^{t}_{-T }\int_{\mathbb{R}^{3}} |u|^{2}
 (\partial_{s}\phi+\Delta \phi)dxds
  + \int^{t}_{-T }
 \int_{\mathbb{R}^{3}}u\cdot\nabla\phi (|u|^{2} +2\Pi)dxds, \label{loc}
 \end{align}
 where non-negative function $\phi(x,s)\in C_{0}^{\infty}(\mathbb{R}^{3}\times (-T,0) )$.\label{SWS3}
\end{enumerate}
\end{definition}

In the light of the natural
scaling property of the time-dependent Navier-Stokes  equations,
we introduce the following dimensionless quantities:
  \be \begin{aligned}&E( \mu)=\mu^{-1}\|u\|^{2}_{L^{\infty,2}(Q(\mu))},
&E&_{\ast}( \mu)=\mu^{-1}\|\nabla u\|^{2}_{L^{2}(Q(\mu))}
,  \nonumber\\
&E_{p}( \mu)=\mu^{p-5}\|u\|^{p}_{L^{p}(Q(\mu))}
,&P&_{5/4}( \mu)= \mu^{-5/4}
\|\nabla\Pi\|^{5/4}_{L^{5/4}(Q(\mu))},\nonumber
\\
& P_{10/7} ( \mu )= \mu^{ -15/7}
\Big\|\Pi-\overline{\Pi}_{B(\mu)}\Big\|^{10/7}_{L^{10/7}(Q(\mu))},
&P&_{5/3} ( \mu )= \mu^{ -5/3}
\Big\|\Pi-\overline{\Pi}_{B(\mu)}\Big\|^{5/3}_{L^{5/3}(Q(\mu))},
\nonumber
\\
&J_{q}( \mu)=  {\mu^{ 2q-5} }  \|\nabla u\|^{q}_{L^{q}(Q(\mu))}.
&~&   ~
\nonumber
\ea\ee
As said before, we need to establish some decay estimates of  scaling invariant quantities to consist with \eqref{GP} for \wblue{$p=q=10/7$.
The first estimate is  partially motivated by}
\cite[Lemma 2.1, p.222]{[WW2]}. We refer the reader to
  \cite{[CY1],[LS],[Lin],[WW1],[CKN]} for different versions.
\begin{lemma}\label{lemma2.1}
For $0<\mu\leq\f{1}{2}\rho$, $7/2\leq b\leq6$ and $3p/5\leq q\leq2(p\geq1,q\geq1)$,~
there is an absolute constant $C$  independent of  $\mu$ and $\rho$,~ such that
\begin{align}
&E_{20/7}( \mu)\leq  C \left(\dfrac{\rho}{\mu}\right)^{ 10/7}
E^{\f{7b-20}{7(b-2)}}( \rho)E^{\f{3b}{7(b-2)}}_{\ast}(  \rho)
    +C\left(\dfrac{\mu}{\rho}\right)
    ^{20/7}E^{10/7}( \rho),\label{inter3}\\
&E_{p}(r)\leq C \left[\B(\f{\rho}{r}\B)^{ \f{p+10-5q}{2}} E ^{\f{(p-q)}{2} }(\rho) J_{q} (\rho)+C\B(\f{r}{\rho}\B)^{p} E^{p/2}(\rho)\right].\label{C3}
\end{align}
\end{lemma}
\begin{proof}
 By utilizing the H\"older  inequality twice  and the Poincar\'e-Sobolev  inequality, for any $20/7<b\leq6$, we infer that
\begin{align}
\int_{B(\mu)}|u-\bar{u}_{B(\rho)}|^{20/7}dx\leq&
C\B(\int_{B(\mu)}|u-\bar{u}_{B(\rho)}|^{2}dx\B)
^{\f{7b-20}{7(b-2)}}
\B(\int_{B(\mu)}|u-\bar{u}_{B(\rho)} |
^{b}dx\B)^{\f{6}{7(b-2)}}\nonumber\\
\leq&C\mu^{\f{3(6-b)}{7(b-2)}}\B(\int_{B(\rho)}|u|^{2}dx\B)
^{\f{7b-20}{7(b-2)}}
 \B(\int_{B(\rho)}|\nabla u|^{2}dx\B)^{\f{3b}{7(b-2)}}. \nonumber
\end{align}
According  to the triangle inequality and the last inequality, we see that
\begin{align}
 \int_{B(\mu)}|u|^{20/7}dx \leq& C\int_{B(\mu)}|u-\bar{u}_{{\rho}}|^{20/7}dx
+C\int_{B(\mu)}|\bar{u}_{{\rho}}|^{20/7} dx \nonumber\\
\leq& C\mu^{\f{3(6-b)}{7(b-2)}}\B(\int_{B(\rho)}|u|^{2}dx\B)
^{\f{7b-20}{7(b-2)}}
 \B(\int_{B(\rho)}|\nabla u|^{2}dx\B)^{\f{3b}{7(b-2)}} \nonumber\\&+
\f{\mu^{3} C}{\rho^{\f{30}{7}}}\B( \int_{B(\rho)}|u|^{2}dx\B)^{10/7}.\nonumber
\end{align}
Integrating \wblue{this inequality in time on $(t-\mu^{2},\,t)$ and utilizing the H\"older} inequality, for any $b\geq7/2$, we get
 \begin{align}
\iint_{Q(\mu)}|u|^{20/7}dxds
\leq& C\mu^{\f{5}{7} }\B(\sup_{t-\rho^{2}\leq s\leq t}\int_{B(\rho)}|u |^{2}dx\B)
^{\f{7b-20}{7(b-2)}}
 \B(\iint_{Q(\rho)}|\nabla u|^{2}dxds\B)^{\f{3b}{7(b-2)}}\nonumber\\&+
 C\f{\mu^{5}}{\rho^{\f{30}{7}}}\B(\sup_{t-\rho^{2 }\leq s\leq t}\int_{B(\rho)}|u|^{2}dx\B)^{10/7},\nonumber
 \end{align}
which yields that
$$
E_{20/7}( \mu ) \leq  C \left(\dfrac{\rho}{\mu}\right)^{ 10/7}
E^{\f{7b-20}{7(b-2)}}( \rho )E^{\f{3b}{7(b-2)}}_{\ast}( \rho )
    +C\left(\dfrac{\mu}{\rho}\right)
    ^{20/7}E^{10/7}( \rho ).
$$
Let us now move to the proof of \eqref{C3}.
Thanks to the H\"older  inequality and the Poincar\'e-Sobolev  inequality, for any  $3p/5\leq q\leq2$, we know that
\begin{align}\nonumber
\int_{B(\mu)}|u-\bar{u}_{B(\rho)}|^{p}dx
\leq& C\B(\int_{B(\mu)}|u-\bar{u}_{B(\rho)}|^{2}dx\B)
^{\f{(p-q)}{2} }
\B(\int_{B(\mu)}|u-\bar{u}_{B(\rho)}|
^{\f{2q}{2+q-p}}dx\B)^{\f{2+q-p}{2}}\\
\leq& C\mu^{\f{5q-3p}{2}}\B(\int_{B(\rho)}|u|^{2}dx\B)
^{\f{(p-q)}{2}}
 \B(\int_{B(\rho)}|\nabla u|^{q}dx\B).\label{rwwi1}
\end{align}
Taking advantage  of the triangle inequality,  the H\"older  inequality and the Poincar\'e-Sobolev  inequality, for any  $3p/5\leq q\leq2$, we know that
\begin{align}\nonumber
\int_{B(\mu)}|u|^{p}dx\leq& C\int_{B(\mu)}|u-\bar{u}_{B(\rho)}|^{p}dx
+C\int_{B(\mu)}|\bar{u}_{B(\rho)}|^{p} dx\\
 \leq& C\mu^{\f{5q-3p}{2}}\B(\int_{B(\rho)}|u|^{2}dx\B)
^{\f{(p-q)}{2} }
 \B(\int_{B(\rho)}|\nabla u|^{q}dx\B) \nonumber\\&+
\f{\mu^{3} C}{\rho^{\f{3p}{2}}}\B( \int_{B(\rho)}|u|^{2}dx\B)^{p/2}.\label{interlast}
 \end{align}
Integrating  this inequality in time on $(-\mu^{2},\,0)$  and using the H\"older inequality,  we obtain
 \begin{align}
\iint_{Q(\mu)}|u|^{p}dxds
\leq& C\mu^{\f{5q-3p}{2}}\B(\sup_{-\rho^{2}\leq s-t\leq0}\int_{B(\rho)}|u |^{2}dx\B)
^{\f{(p-q)}{2} }
 \B(\iint_{Q(\rho)}|\nabla u|^{q}dxds\B)\nonumber\\&+
 C\f{\mu^{5}}{\rho^{\f{3p}{2}}}\B(\sup_{-\rho^{2 }\leq s-t \leq0}\int_{B(\rho)}|u|^{2}dx\B)^{p/2},\nonumber
 \end{align}
which in turn implies that
$$
E_{p}(r)\leq C \left[\B(\f{\rho}{r}\B)^{ \f{p+10-5q}{2}} E ^{\f{(p-q)}{2} }(\rho) J_{q} (\rho)+C\B(\f{r}{\rho}\B)^{p} E^{p/2}(\rho)\right].
$$
 This achieves the proof of this lemma.
\end{proof}
In
the spirit of   [17, Lemma 2.1, p.222], we can make full use of the interior estimate of harmonic function to establish the following   decay estimate of pressure   $\Pi-\overline{\Pi}_{ B (r)}$.
The pressure  $\Pi$ in terms of
$\nabla \Pi$  in equations enables us to apply
  this lemma in the proof of Theorem \ref{the1.2} and Theorem \ref{main}.
\begin{lemma}\label{presure}
For $0<\mu\leq\f{1}{8}\rho$, there exists an absolute constant $C$  independent of $\mu$ and $\rho$ such that
\begin{align}
&P_{10/7}( \mu)
\leq C\left(\f{\rho}{\mu}\right)^{15/7}E_{20/7}( \rho)+
C\left(\f{\mu}{\rho}\right)^{16/7}P_{10/7}( \rho),\label{ppp} \\
&P_{p/2}(\mu)\leq C \left[\B(\f{\rho}{\mu}\B)^{5-p}E^{\f{(p-q)}{2} }(\rho) J_{q}(\rho)+\B(\f{\mu}{\rho}\B)^{\f{3p-4}{2}}  P_{p/2} (\rho)\right],\label{P3/22}\end{align}
where $p$ and $q$ are defined in Lemma \ref{lemma2.1}.
\end{lemma}
\begin{proof}
We choose  the usual smooth cut-off function $\phi\in C^{\infty}_{0}(B(\rho/2))$ such that $\phi\equiv1$ on $B(\f{3}{8}\rho)$ with $0\leq\phi\leq1$ and
$|\nabla\phi |\leq C\rho^{-1},~|\nabla^{2}\phi |\leq
C\rho^{-2}.$

It follows from  divergence free condition   that
$$
\partial_{i}\partial_{i}(\Pi\phi)=-\phi \partial_{i}\partial_{j}\big[U_{i,j}\big]
+2\partial_{i}\phi\partial_{i}\Pi+\Pi\partial_{i}\partial_{i}\phi
,$$
where $U_{i,j}=(u_{j}- \overline{u_{j}} _{\rho/2})(u_{i}-\overline{u_{i}}_{\rho/2})$.\\
For any $x\in B(\f{3}{8}\rho)$, we derive from integrations by parts that
\begin{align}
\Pi(x)=&K \ast \{-\phi \partial_{i}\partial_{j}[U_{i,j}]
+2\partial_{i}\phi\partial_{i}\Pi+\Pi\partial_{i}\partial_{i}\phi
\}\nonumber\\
=&-\partial_{i}\partial_{j}K \ast (\phi [U_{i,j}])\nonumber\\
&+2\partial_{i}K \ast(\partial_{j}\phi[U_{i,j}])-K \ast
(\partial_{i}\partial_{j}\phi[U_{i,j}])\nonumber\\
& +2\partial_{i}K \ast(\partial_{i}\phi \Pi) -K \ast(\partial_{i}\partial_{i}\phi \Pi)\nonumber\\
=: &P_{1}(x)+P_{2}(x)+P_{3}(x),\label{pp}
\end{align}
where $K$   represents the standard normalized fundamental solution of Laplace equation.
Thanks to  $\phi(x)=1$ ($x\in B(\rho/4$)), we know that
\[
\Delta(P_{2}(x)+P_{3}(x))=0.
\]
In the light of the interior  estimate of harmonic function
and   H\"older's inequality, we see that, for every
$ x_{0}\in B(\rho/8)$,
$$\ba
|\nabla (P_{2}+P_{3})(x_{0})|&\leq \f{C}{\rho^{4}}\|(P_{2}+P_{3})\|_{L^{1}(B_{x_{0}}(\rho/8))}
\\
&\leq \f{C}{\rho^{4}}\|(P_{2}+P_{3})\|_{L^{1}(B(\rho/4))}\\
&\leq \f{C}{\rho^{(p=6)/p}}\|(P_{2}+P_{3})\|_{L^{p/2}(B(\rho/4))}
,
\ea$$
which in turn implies
$$\|\nabla (P_{2}+P_{3})\|^{p/2}_{L^{\infty}(B(\rho/8))}\leq C \rho^{-(p+6)/2}\|(P_{2}+P_{3})\|^{10/7}_{L^{p/2}(B(\rho/4))}.$$
This combined with   the  mean value theorem yields that, for any  $\mu\leq \f{1}{8}\rho$,
$$\ba
\|(P_{2}+P_{3})-\overline{(P_{2}+P_{3})}_{B(\mu)}\|^{p/2}_{L^{p/2}(B(\mu))}\leq&
C\mu^{3} \|(P_{2}+P_{3})-\overline{(P_{2}+P_{3})}_{B(\mu)}\|^{p/2}_{L^{\infty}(B(\mu))}\\
\leq& C
\mu^{(p+6)/2} \|\nabla (P_{2}+P_{3})\|^{p/2}_{L^{\infty}(B(\rho/8))}\\
\leq& C\Big(\f{\mu}{\rho}\Big)^{(p+6)/2}\|(P_{2}+P_{3})\|^{p/2}_{L^{p/2}
(B(\rho/4))}.
\ea$$
Note that $(P_{2}+P_{3})-\overline{(P_{2}+P_{3})}_{B(\rho/4)}$ is also a harmonic function  on $B(\rho/4)$, hence, there  holds
$$\ba
&\|(P_{2}+P_{3})-\overline{(P_{2}+P_{3})}_{B(\mu)}\|
^{p/2}_{L^{p/2}(B(\mu))}
\\
\leq & C\Big(\f{\mu}{\rho}\Big)^{(p+6)/2}
\|(P_{2}+P_{3})-\overline{(P_{2}+P_{3})}_{B(\rho/4)}\|
^{p/2}_{L^{p/2}(B(\rho/4))}.
\ea$$
By   the triangle inequality, we deduce that
$$\ba
&\|(P_{2}+P_{3})-\overline{(P_{2}+P_{3})}_{B(\rho/4)}\|_{L^{p/2}(B(\rho/4))}\\
\leq& \|\Pi-\overline{\Pi}_{B(\rho/4)}\|_{L^{p/2}(B(\rho/4))}
+\|P_{1}-\overline{P_{1}}_{B(\rho/4)}\|_{L^{p/2}(B(\rho/4))}
\\
\leq& C\|\Pi-\overline{\Pi}_{B(\rho)}\|_{\wblue{L^{p/2}(B(\rho/4))}}
+C\|P_{1}\|_{L^{p/2}(B(\rho/4))},
\ea$$
which tells us that
\begin{align}
&\|(P_{2}+P_{3})-\overline{(P_{2}+P_{3})}_{B(\mu)}\|
^{p/2}_{L^{p/2}(B(\mu))}\nonumber\\
\leq& C\Big(\f{\mu}{\rho}\Big)^{(p+6)/2}\Big(\|\Pi-\overline{\Pi}
_{B(\rho)}\|^{p/2}_{L^{p/2}(B(\rho))}
+\|P_{1}\|^{p/2}_{L^{p/2}(B(\rho/4))}\Big).
\label{p2rou}\end{align}
The classical Calder\'on-Zygmund theorem ensures that
\begin{align}
\int_{B(\rho/4)}|P_{1}(x)|^{p/2}dx
\leq  C \int_{B(\rho/2)}|u-\bar{u}_{B(\rho/2)}|^{p}dx,
  \label{lem2.4.2}
\end{align}
from which it follows that, for any  $\mu\leq \f{1}{8}\rho$,
\begin{align}
\int_{B(\mu)}|P_{1}(x)|^{p/2}dx \leq C \int_{B(\rho/2)}|u-\bar{u}_{B(\rho/2)} |^{p}dx.
\label{lem2.4.3}\end{align}
Employing time integration on $(t-\mu^{2}, t)$ and  the triangle inequality, we conclude using \eqref{p2rou}-\eqref{lem2.4.3}  that
\begin{align}
&\iint_{Q(\mu)}|\Pi-\overline{\Pi}_{B(\mu)}|^{p/2}dxds\nonumber\\
\leq& \iint_{Q(\mu)}|P_{1}-\overline{P_{1}}_{B(\mu)}|^{p/2}dxds+
\iint_{Q(\mu)}|P_{2}+P_{3}-\overline{(P_{2}+P_{3})}_{B(\mu)}|^{p/2}dxds
\nonumber\\ \leq & C\iint_{Q(\mu)}|P_{1}|^{p/2}dxds+C\Big(\f{\mu}{\rho}\Big)^{(p+6)/2}\Big(\|\Pi-\overline{\Pi}
_{B(\rho)}\|^{p/2}_{L^{p/2}(B(\rho))}
+\|P_{1}\|^{p/2}_{L^{p/2}(B(\rho/4))}\Big)
\nonumber\\ \leq &C \iint_{Q(\rho/2)}|u|^{p}dxds+C\Big(\f{\mu}{\rho}\Big)^{(p+6)/2}
\|\Pi-\overline{\Pi}
_{B(\rho)}\|^{p/2}_{L^{p/2}(B(\rho))},\label{pressuresan}
\end{align}
which means \wblue{\eqref{ppp}}.
 A slight modified the above the proof of the latter inequality together with \eqref{rwwi1}
gives \eqref{P3/22}.
The proof of this lemma is   completed.
\end{proof}

\section{Proof  of Theorem \ref{the1.2}}
\label{sec3}
\setcounter{section}{3}\setcounter{equation}{0}
The main part of this sections is the proof of Theorem \ref{the1.2}. The method follows closely the recent developments in \cite{[KY6],[WW2]}. The main ingredient is to apply \eqref{GP} and decay-type estimates established  in \wblue{Section 2}.
\begin{proof}[Proof of Theorem \ref{the1.2}]
From \eqref{cond}, we choose $2\rho<1$ such that $\rho^{\beta}<1/2$, where $\beta$ will be
determined  later   and
\be\label{assume}
 \iint_{Q (2\rho)}
|\nabla u |^{2} +| u |^{ 10/3}+|\Pi-\overline{\Pi}_{B(2\rho)} |^{ 5/3}+
 |\nabla \Pi| ^{5/4}dxds \leq    (2\rho)^{ 5/3-\gamma}\varepsilon_{1}.
\ee
First, one can derive   \be E(\rho)\leq C\varepsilon_{1}^{3/5}\rho^{-\f{3\gamma}{5}},~~ (\gamma\leq5/12),\label{E}\ee    from \eqref{assume} via the local energy inequality \eqref{loc}, which is
 proved in \cite{[WW2]}. Here we omit the details, see \cite[Proof of theorem 1.2, p.1768-1769]{[WW2]} for details.
Second, iterating (\ref{ppp}) in Lemma \ref{presure}, we see  that
\be\label{referee}
P_{10/7}( \theta^{N}\mu)\leq C\sum^{N}_{k=1}\theta^{-\f{15}{7}+\f{16(k-1)}{7}}
E_{20/7}( \theta^{N-k}\mu)+C\theta^{16N/7}
P_{10/7}( \mu).
\ee
With the help of   the Poincar\'e-Sobolev  inequality and H\"older's inequality, we get
\begin{align}
\nonumber
 \|\Pi-\overline{\Pi}_{ B(\mu)}\|^{10/7}_{L^{10/7}(Q( \mu))} &\leq
 \|\Pi-\overline{\Pi}_{ B(\mu)}\|^{5/7}_{L^{5/4,15/7}(Q( \mu))} \|\Pi-\overline{\Pi}_{ B(\mu)}\|^{5/7}_{L^{5/3,15/14}(Q( \mu))}\nonumber\\
 &\leq
 C\mu^{5/7}\|\nabla \Pi \|^{5/7}_{L^{5/4}(Q( \mu))} \|\Pi-\overline{\Pi}_{ B(\mu)}\|^{1/2}_{L^{5/3}(Q( \mu))},
 \label{pressureinti}\end{align}
Dividing both sides of the last inequality by $\mu^{15/7}$, we arrive at
$$
P_{10/7}( \mu)\leq CP^{4/7}_{5/4}( \mu)
P^{3/7}_{5/3}( \mu).
$$
We substitute the above inequality into \eqref{referee} to obtain that
\be\label{refer}
P_{10/7}( \theta^{N}\mu)\leq C\sum^{N}_{k=1}\theta^{-\f{15}{7}+\f{16(k-1)}{7}}
E_{20/7}( \theta^{N-k}\mu)+C\theta^{16N/7}
P^{4/7}_{5/4}( \mu)
P^{3/7}_{5/3}( \mu).
\ee
To proceed further, we set $r=\rho^{\alpha}=\theta^{N}\mu$,~$ \theta=\rho^{\beta}$, ~$r_{i}=\mu =\theta^{-i}r=\rho^{\alpha-i\beta}(1\leq i\leq N)$, where $\alpha$ and $\beta$ are determined by $\gamma$.
 Their precise selection will be given
   in the end.
 Hence, we derive from \eqref{refer} that
\be\ba
&\wblue{P_{10/7}( r)+E_{20/7}( r)}\\
\leq&
C\sum^{N}_{{k=1}}\theta^{-\f{15}{7}+\f{16(k-1)}{7}}
 E_{20/7}( r_{k}) +C\theta^{16N/7}P^{4/7}_{5/4}
( r_{N})
P^{3/7}_{5/3}( r_{N})\\
:=&I+II,\label{key}
\ea\ee
where we have used the fact that \wblue{$E_{20/7}(u,r)\leq C\theta^{-\frac{15}{7}}E_{20/7}(u,\theta^{-1}r)$}.
Our aim below is to resort to \eqref{GP} to complete the proof, that is, there exists a constant $r>0$ such that $P_{10/7}( r)+E_{20/7}( r)<\varepsilon_0$.
To this end,
we   adopt \eqref{inter3} with \wblue{$b=7/2$} in Lemma \ref{lemma2.1}, \eqref{E} and   \eqref{assume} to obtain
$$\ba
E_{20/7}( r_{k})&\leq C \Big(\f{\rho}{ r_{k}}\Big)^{\f{10}{7}}E^{3/7}( \rho)
 E_{\ast} ( \rho)+C\Big(\f{ r_{k}}{\rho}\Big)^{20/7}E^{10/7}( \rho)\\
  &\leq C\varepsilon_{1}^{6/7} \Big( \rho^{\f{44}{21}-\f{10}{7}(\alpha-k\beta)-\f{44\gamma}{35}}
  +\rho^{\f{20}{7}\alpha-\f{20}{7}-\f{20}{7}k\beta-\f{6\gamma}{7}} \Big).
 \ea$$
Substituting  the last inequality into $I$ produces that
$$\ba
I&\leq C\varepsilon_{1}^{6/7}\sum^{N}_{k=1}\B(\rho^{-\f{31\beta}{7}+\f{26 k\beta}{7}-\f{10\alpha}{7}-\f{44\gamma}{35}+\f{44}{21}}
  +\rho^{-\f{31\beta}{7}+\f{20\alpha}{7}-\f{20}{7}-\f{4 k\beta}{7}-\f{6\gamma}{7}} \B).\ea $$
  To minimise the righthand side of this inequality, we choose
\be\label{aerfa}
\alpha=\f{7}{30}(\f{26\beta}{7}+\f{104}{21}-\f{2\gamma}{5}+\f{4 N\beta}{7})
\ee
to conclude that, for sufficiently large $N$,
 \be\ba \label{key1}
   I &\leq C\varepsilon_{1}^{6/7}\B(
    \rho^{-\f{5\beta}{7}-\f{10\alpha}{7}-\f{44\gamma}{35}+\f{44}{21}}
  +\rho^{-\f{31\beta}{7}+\f{20\alpha}{7}-\f{20}{7}-\f{4 N\beta}{7}-\f{6\gamma}{7}}\B)
  \\
    &\leq C\varepsilon_{1}^{6/7} \rho^{-\f{41\beta}{21}+\f{4}{9}
    -\f{118\gamma}{105}-\f{4N\beta}{21}}.
  \ea \ee
  To bound $II$, we will temporarily assume  that $r_{N}\leq \rho$, namely
 \be
\rho^{\alpha-N\beta}\leq\rho.\label{c3}
\ee
Combining \eqref{assume}
and \eqref{aerfa}, we see that
\be\ba\label{key2}
II \leq&  C\rho^{\f{16N\beta}{7}} r_{N}^{-\f{10}{7}}
\B(\iint_{Q(2\rho)}|\nabla\Pi|^{5/4}dxds\B)^{4/7}
\B( \iint_{Q(r_{N})}|\Pi-\overline{\Pi}_{2\rho}|^{5/3}dxds\B)^{3/7}\\
\leq&  C\rho^{\f{16N\beta}{7}}r_{N}^{-\f{10}{7}}
\B(\iint_{Q(2\rho)}|\nabla\Pi|^{5/4}dxds\B)^{4/7}
\B(\iint_{Q(2\rho)}|\Pi-\overline{\Pi}_{2\rho}|^{5/3}dxds\B)^{3/7}\\
\leq&  C\rho^{\f{74 N\beta }{21}+\f{1}{63} -\f{13\gamma}{15}-\f{26\beta}{21}}\varepsilon_{1}.
 \ea\ee
In order to conclude that $
I+II
 \leq C\varepsilon_{1}^{6/7}\leq \varepsilon_0$,
  we need  $-\f{41\beta}{21}+\f{4}{9}
    -\f{118\gamma}{105}-\f{4N\beta}{21}\geq 0$ and $\f{74 N\beta }{21}+\f{1}{63} -\f{13\gamma}{15}-\f{26\beta}{21}\geq 0$. In addition, it follows from \eqref{c3}  that  $\alpha-N\beta-1\geq0.$ Hence, we sum up all the restrictions of $\gamma$ below
\be\label{last1}
\gamma\leq \min\Big\{
\f{5(28-12N\beta-123\beta)}{354},
\f{5(1+222N\beta-78\beta)}{273},
\f{5(7-39N\beta+39\beta)}{21},\f{5}{12}
\Big\}.
\ee
Maximising this bound on $\gamma$ with respect to $N\beta$, we obtain $N\beta=135/1516$. Furthermore, it follows \eqref{last1} from that
$$
\beta=\f{135}{1516N}\leq\wblue{\f{118}{205}}\B(\f{865}{2274}-\gamma\B).
$$
 Hence, choosing $\beta$ sufficiently small by selecting $N$
sufficiently large,
we can have any
 $\gamma<865/2274$.
Then, we pick $\alpha=
\f{7}{30}(\frac{26}{7}\beta-\f{2\gamma}{5}+\f{39821}{7959})$. In this stage,
from \eqref{key}, \eqref{key1} and \eqref{key2}, we get
$$
  P_{10/7}(r)+E_{20/7}(r)
 \leq C\varepsilon_{1}^{6/7}< \varepsilon_0, $$
with $r=\rho^{\alpha}$. By \eqref{GP}, we know that $(x,t)$ is a regular point in this stage.
This completes the proof of Theorem \ref{the1.2}.
\end{proof}



\section{Proof of Theorem \ref{main}}
In the spirit of \cite{[Choe],[CL]}, we begin with some technical lemmas for the proof of Theorem \ref{main}. These lemmas are  parallel to the one of \cite{[CL]}.
It is worth remarking that the proof of Lemma \ref{J-b control} is slightly  different from the ones in \cite{[Choe],[CL],[CY1],[CY2]}.
  In what follows, we set \wblue{$m(r)=(\Gamma(r))^{\sigma}=(\log(e/r))^{\sigma}$}, where $\sigma\in(0,1)$ will be determined later.

Before going further, we set
  $$
  F(m)=\B\{(x,t)\B |\limsup_{r\rightarrow0}\f{E_{\ast}(r)}{m(r)}\leq1 \B\}.
  $$
\begin{lemma}\label{A-1 control}
Assume that $(x,t)\in F(m)\cap\mathcal{S}$ and the pair  $(p,\,q)$ is used in Lemma \ref{lemma2.1}. Then, there exists a positive constant $c_1$ and  $c_2$	independent of $(x,t)$
 such that
\begin{align}
&\limsup_{r\rightarrow 0}\f{ E( r)}{m^2(r)}\leq c_1,\label{CL}\\
&\limsup_{r\rightarrow 0}\f{ P_{p/2}( r)}{m^{p-1}(r)}\leq c_2.\label{w}
\end{align}
\begin{proof}
The   reader is referred to \cite[Lemma 1, page 357]{[CL]} for the detail  of \eqref{CL}. We outline the proof of \eqref{w}.
 Let $g(r)=\f{ P_{p/2}( r)}{m^{p-1}(r)}$, from \eqref{C3} with $q=2$, we see that
$$\ba
g(\mu)&\leq C \left[\B(\f{\rho}{\mu}\B)^{5-p}\B(\f{E(\rho) }{m^{2}(\rho)} \B)^{\f{(p-2)}{2} }  \f{E_{\ast}(\rho)}{m(\rho)}+\B(\f{\mu}{\rho}\B)^{\f{3p-4}{2}}  g (\rho)\right]
\\
&\leq C \left[\B(\f{\rho}{\mu}\B)^{5-p}+\B(\f{\mu}{\rho}\B)^{\f{3p-4}{2}}  g (\rho)\right],
\ea$$
where we have utilized the hypothesis and  \eqref{CL}. This together with the iteration method see, (e.g. \cite{[Seregin1]}) allows us to obtain \eqref{w}.
\end{proof}
\end{lemma}
\begin{lemma}\label{J-b control}
Let  $(x,t)\in F(m)\cap\mathcal{S}$.  Then,  there exists a positive constant $c_2$ independent of $(x,t)$ such that
$$\liminf_{r\rightarrow 0}J_q(r )  m(r)^{\tau} \geq c_3,$$
where $\tau=\f{p^{2}+(6-3q)p+4q}{3p-4}$ and the pair $(p,\,q)$ is  utilized in Lemma \ref{lemma2.1}.
\end{lemma}
\begin{proof}
Assume that the statement fails, then, for any $\eta>0$, there exists a singular point $(x,t)$ and a sequence $r_n\rightarrow 0$ such that
\be\label{l4.2.1} J_q(r_n )  m(r_n)^{\tau} < \eta.\ee
It follows from  \eqref{C3}, Lemma \ref{A-1 control} and \eqref{l4.2.1}  that, for $\theta_n<1/8$,
\be\ba
E_{p}(\theta_n r_n )+P_{p/2}(\theta_n r_n )\leq& C \theta_{n} ^{p}m^{p}(r_n)+C \theta_{n}^{-\f{p+10-5q}{2}}m(r_n)^{p-q} J_q(r_n )\\&+C\theta_{n}^{\f{3p-4}{2}}m^{p-1}(r_n)+C \theta_{n}^{-(5-p)}m(r_n)^{p-q}J_q(r_n )\\
\leq& C\theta_{n}^{\f{3p-4}{2}}m^{p}(r_n)+C \theta_{n}^{-(5-p)}m(r_n)^{p-q}J_q(r_n )\\
\leq& C m(r_n)^{-\f{q(3p-4)}{6+p}+p}J_q(r_n)^{\f{3p-4}{6+p}}\\
\leq& c \eta^{\f{3p-4}{6+p}},
\ea \label{A-1,A-2-control1}\ee
where
$\theta_n=m(r_n)^{-2q/(6+p)}J_q^{2/(6+p)}(r_n )$.
Note that  $\theta_n$ goes to $0$ as $n\rightarrow \infty$ by \eqref{l4.2.1}.
Let $\rho_n=\theta_n r_n$ and $\epsilon_2=c \eta^{\f{3p-4}{6+p}}$ such that  $\varepsilon_2<\min\{1,\varepsilon^{10/7}_0/2\}$.
For
sufficiently large $n$, we see that
$$E_{p}(\theta_n r_n )+P_{p/2}(\theta_n r_n )\leq \varepsilon_2.$$
This together with \eqref{wwww} implies that  $(x,t)$ is a regular point. Thus, we reach a contradiction and finish the proof.
\end{proof}
\begin{coro}\label{J-b control-1coro}
Suppose that $(x,t)\in F(m)\cap\mathcal{S}$ and     the pair  $(p,\,q)$ is defined in Lemma \ref{lemma2.1}, then, there exists a small constant $c_{4}$ such that
$$\liminf_{r\rightarrow 0} \widetilde{J}_q(r )  m(r)^{\tau} \geq c_3/2,$$
where
$\tilde{J}_q(r )=r^{2q-5}\iint _{\tilde{Q}( r)\cap \{(x,t) |\nabla u(x,t)| > c_4 r^{-2} m(r)^{\frac{-\tau}{q}}\}}  |\nabla u|^q dxds$  and $c_{3}$ is defined as in Lemma \ref{J-b control}.
\end{coro}
\begin{proof}
After a straightforward \wred{computation}, we get
\begin{align*}
J_q(r ) -\widetilde{J}_q(r )&=r^{2q-5}\iint _{\tilde{Q}( r)\cap \{(x,t)| |\nabla u(x,t)| \leq c_4 r^{-2} m(r)^{\frac{-\tau}{q}}\}}  |\nabla u|^q  dxds\\
&\leq c r^{\wred{2q-5}} c_4 r^{\wred{-2q}} m(r)^{-\tau} r^5=c c_3  m(r)^{-\tau},
\end{align*}
which yields that
\begin{equation}\label{limit}
  \limsup_{r\rightarrow 0}  m(r)^{\frac{27-5q}{5}}[J_q(r ) -\widetilde{J}_q(r )]\leq c c_4.
\end{equation}
Combining  \eqref{limit} and  Lemma \ref{J-b control} ensures that
\begin{align*}
 \liminf_{r\rightarrow 0} \widetilde{J}_q(r)  m(r)^{\frac{27-5q}{5}}&\wred{\geq}
 \liminf_{r\rightarrow 0} J_q(r)  m(r)^{\frac{27-5q}{5}}+\liminf_{r\rightarrow 0} [\widetilde{J}_q(r)- J_q(r)]  m(r)^{\frac{27-5q}{5}}\\
&\geq c_3-cc_4\\
&\geq c_3/2,
\end{align*}
where $c_4=c_3/2c$. This concludes the proof of this lemma.
\end{proof}
Now we are in a position to show Theorem \ref{main}.
\begin{proof}[Proof of Theorem \ref{main}]
Let $G_k$ denote the set of $(x,t)\in F(m) \cap \mathcal{S}$ such that
\begin{equation}\label{order}
   c_2/4\leq m(r)^{\tau} \widetilde{J}_q(r)\quad  \text{and} \quad E_{\ast}(r) \leq 2 m(r),
\end{equation}
for $ 0<r<\f{1}{k}$. From Corollary \ref{J-b control-1coro},
we know that $F(m) \cap \mathcal{S}=\bigcup_{k=1}^\infty G_k=\lim\limits_{k\rightarrow \infty} G_k$.
Let $r_0={1/{k}}$, then it follows from \eqref{order}   that
\begin{equation}\label{order1}
  \quad \wred{E_{\ast}(r_1)} \leq c\ m(r_1)m(r_2)^{\tau} \wred{\widetilde{J}_q(r_2)},
\end{equation}
for any $0<r_1, r_2< r_0$.

We denote $d^{k}(x,t)=\inf \{|x-y|+|t-s|^{\f{1}{2}}: (y,s)\in G_k \}$ and
define the neighbourhood of $G_{k}$ by  $L^{k}(r)= \{(x,t)|~ d(x,t)<r \}, ~ \widetilde{L}^{k}(r)=L^{k}(r)\cap \widetilde{K}(r)$, where $\widetilde{K}(r)=\{(x,t): |\nabla u(x,t)| > c_3 r^{-2} m(r)^{\frac{-\tau}{q}}\}$. By the classical Vitali covering lemma, $G_{k}\subset \mathcal{S}$ and \eqref{ckn}, we know that
there is a sequence of parabolic cylinders $\{Q (x_{i},t_{i};\, r)\}$ such that
\begin{align}
&G_k\subset \bigcup_i Q(x_{i},t_{i};\, {5r}),\nonumber\\
&(x_{i},t_{i}) \in G_k,\nonumber\\
&Q(x_{m},t_{m};\, {r})\cap Q(x_{n},t_{n};\, {r})=\varnothing, \ m\neq n,\nonumber\\
&r\leq \varepsilon^{-1} \iint _{Q(x_i,t_i;\,{r})}|\nabla u|^2 dxds. \label{refree212}
\end{align}
Moreover, we would like to point out that the  radius $r$  in $\{Q (x_{i},t_{i};\, r)\}$ above is  independent on the points $(x_{i},t_{i})$, which can be   examined by Vitali covering lemma. For this fact, see also \cite[Proof of Theorem B, p.807]{[CKN]} and
\cite[Proof of theorem 2.1 assuming theorem 2.2, p.2892]{[Kukavica]}.

Thanks to  the definition of $L^{k}(r)$, we infer that
$$ L^{k}(r)\subset \bigcup_i Q(x_{i},t_{i};\, 6r),$$
which yields that
\begin{equation*}
  \iint _{L^{k}(r)} |\nabla u|^2 dxds \leq \sum_i \iint _{Q(x_{i},t_{i};\, {6r})} |\nabla u|^2 dxds.
\end{equation*}
By \eqref{order1}, for $0<r<r_{0}$, we arrive at that
 \begin{align}\label{E-control-by-E-1}
\begin{aligned}
  \iint _{L^{k}(r)} |\nabla u|^2 dxds &\leq C r^{2q-4}m(6r)m(r)^{\tau}\sum_i \iint _{Q(x_{i},t_{i};\,r)\cap \tilde{K}(r )} |\nabla u|^q dxds\\
&=C r^{2q-4}m(r)^{\tau+1} \iint _{\widetilde{L}(r)} |\nabla u|^q dxds.
\end{aligned}
\end{align}
 Define  $d^{k}_n(x,t)=\max \{d^{k}(x,t), {\f{1}{n}}\}$ with $n>k$.
Multiplying $\eqref{E-control-by-E-1}$ by
$r^{-1}$ and integrating \wred{the obtained inequality} over ($n^{-1}$, $r_0$),
we get
\begin{align}\label{interchange}
 \begin{aligned}
 \iint _{L^{k}(r_0)} [\Gamma(d_n)-\Gamma(r_0)] |\nabla u|^2 dxds&=\int _{n^{-1}}^{r_0} \iint _{L^{k}(r)}  r^{-1} |\nabla u|^2 dxdsdr\\
 &\leq C \int _{n^{-1}}^{r_0} \wred{r^{2q-5}}\Gamma(r)^{(\tau+1)\sigma} \iint _{\widetilde{L}^{k}(r)} |\nabla u|^q dxdsdr,
\end{aligned}
\end{align}
where we used the definition of $\Gamma(r).$

Thanks to Tonelli's theorem, we interchange the order of integration for the right-hand side of the inequality $\eqref{interchange}$ to arrive at
\begin{align*}
   &\iint _{L^{k}(r_0)} [\Gamma(d_n)-\Gamma(r_0)] |\nabla u|^2 dxds\\
   \leq &C  \iint _{L^{k}(r_0)} |\nabla u|^q
\int _{n^{-1}}^{r_0}\chi_{\widetilde{K}(r)\cap L^{k}(r)}(x,t) r^{2q-5}\Gamma(r)^{(\tau+1)\sigma} drdxds\\
\leq &C \iint _{L^{k}(r_0)} |\nabla u|^q
\int _{n^{-1}}^{r_0}\min\{\chi_{\widetilde{K}(r)}(x,t) ,\chi_{L^{k}(r)}(x,t) \} r^{2q-5}\Gamma(r)^{(\tau+1)\sigma} drdxds.
\end{align*}
Due to the properties of $\Gamma(r)$ and the definition of
\wred{$L^{k}(r)$}, for $q<2$, we find
\begin{equation}\label{first-integration}
\int _{n^{-1}}^{r_0}\chi_{\wred{L^{k}(r)}}(x,t)  r^{2q-5}\Gamma(r)^{(\tau+1)\sigma} dr\leq C
d_n^{2q-4}\Gamma(d_n)^{(\tau+1)\sigma}.
\end{equation}
For $(x,t)\in \widetilde{K}(r)$, it is clear that
\begin{equation}\label{KRZ}
r^{-1}\leq c_{3}^{-\frac{1}{2}}|\nabla u(x,t)|^{\frac{1}{2}}\wred{\Gamma(r)^{\frac{\tau\sigma}{2q}}}.
\end{equation}
In the light of  $\lim\limits_{r\rightarrow 0} r\Gamma(r)=0$, it turns out that
\begin{equation*}
  \Gamma(r)\leq \frac{C}{r}.
\end{equation*}
Consequently, we can obtain that
\begin{equation*}
  r^{-1}\leq C|\nabla u(x,t) |^{\frac{1}{2}}r^{\frac{-\tau\sigma}{2q}},
\end{equation*}
which in turn implies
\begin{equation*}
  r^{-1}\leq C|\nabla u(x,t) |^{\frac{1}{2(1-\delta)}},
\end{equation*}
where $\delta=\frac{\tau\sigma}{2q}\in(0,1)$.
With the help of  the properties of $\Gamma(r)$, we infer that
\begin{equation}\label{GRNZ}
  \Gamma(r)\leq \Gamma(C|\nabla u(x,t) |^{\frac{-1}{2(1-\delta)}})\leq C(\delta)\Gamma(|\nabla u(x,t) |^{-\frac{1}{2}}).
\end{equation}
Combining this and    \eqref{KRZ}, we get the following result
\begin{equation}\label{RNZ}
  r^{-1}\leq C|\nabla u(x,t) |^{\frac{1}{2}}\Gamma(|\nabla u(x,t) |^{-\frac{1}{2}})^{\delta}.
\end{equation}
From \eqref{GRNZ} and \eqref{RNZ}, for $3p/5\leq q<2$, we see that
\begin{align} \label{second-integration}
 \begin{aligned}
 &\int _{n^{-1}}^{r_0}\chi_{\widetilde{K}}(x,t) r^{2q-5}\Gamma(r)^{(\tau+1)\sigma} dr \\
 \leq &C
 \int _{n^{-1}}^{r_0}\chi_{\widetilde{K}}(x,t)  r^{2q-5} dr \Gamma( |\nabla u(x,t) |^{-{\f{1}{2} }})^{(\tau+1)\sigma}\\
\leq &C |\nabla u(x,t) |^{2-q}\Gamma(|\nabla u(x,t) |^{-{\f{1}{2} }})^{(4-2q)\delta+(\tau+1)\sigma}\\
= &C |\nabla u(x,t) |^{2-q}\Gamma(|\nabla u(x,t) |^{-{\f{1}{2}}})^{\f{(2-q)\tau\sigma}{q}+(\tau+1)\sigma}.
\end{aligned}
\end{align}
It follows from \eqref{first-integration} and \eqref{second-integration} that
\begin{align}\label{inequality}
\begin{aligned}
  &\iint _{L^{k}(r_0)} [\Gamma(d_n)-\Gamma(r_0)] |\nabla u|^2 dxds\\
\leq & C \iint _{L^{k}(r_0)} |\nabla u|^q
\min\{{d_n^{2q-4}}\Gamma(d_n)^{\delta_{1}},\ |\nabla u(x,t)|^{2-q}
\Gamma(|\nabla u(x,t)|^{-{\f{1}{2}}})^{\delta_{2}}\}dxds,
\end{aligned}
\end{align}
where
\be \label{extend}
 \delta_{1}=(\tau+1)\sigma \quad ~~\text{and}~~~ \delta_{2}=\f{(2-q)\tau\sigma}{q}+(\tau+1)\sigma.
\ee
By $\sigma<27/113$, we can choose   $q$  sufficiently  close to $2$  and $p$  sufficiently  close to $10/3$  to  guarantee that
 \be\label{wu2} \delta_{2}=\f{(2-q)\tau\sigma}{q}+(\tau+1)\sigma<1.\ee
 In case   $ |\nabla u|\geq d_n^{-2}$,  we see that
$$
  |\nabla u|^q\min\{d_n^{2q-4}\Gamma(d_n)^{\delta_{1}}, |\nabla u(x,t)|^{2-q}\Gamma(|\nabla u(x,t)|^{-{\f{1}{2}}})^{\delta_{2}}\}
  \leq  |\nabla u|^2 \Gamma(d_n)^{\delta_{1}}\leq  |\nabla u|^2 \Gamma(d_n)^{\delta_{2}}.
$$
Otherwise, if $|\nabla u|<d_n^{-2}$, we get
$$
  |\nabla u|^q\min\{d_n^{2q-4}\Gamma(d_n)^{\delta_{1}}, |\nabla u(x,t)|^{2-q}\Gamma(|\nabla u(x,t)|^{-{\f{1}{2}}})^{\delta_{2}}\}
  \leq  |\nabla u|^2 \Gamma(d_n)^{\delta_{2}}.
$$
So, no matter in which case, we always  choose
  $r_0$ sufficiently small to get
\begin{equation*}
 C \ |\nabla u|^2 \Gamma(d_n)^{\delta_{2}}\leq
  {\f{1}{4}} |\nabla u|^2 \Gamma(d_n).
\end{equation*}
This together with \eqref{inequality} implies that
\begin{equation*}
  \iint _{L^{k}(r_0)} \Gamma(d_n)|\nabla u|^2 dxds \leq c(\sigma, q, r_0)<\infty.
\end{equation*}
We deduce from  monotone convergence theorem  in the last inequality that
\begin{equation}
  \iint _{L^{k}(r_0)} \Gamma(d)|\nabla u|^2 dxds <\infty.\label{ref}
\end{equation}
 Since
$\widetilde{K}(r)=\{(x,t): |\nabla u(x,t)| > c_4 r^{-2} m(r)^{\frac{-\tau}{q}}\},$
by means of Chebyshev's inequality, we infer that
$$\iint _{Q(x_{i},t_{i};\,r)\cap \tilde{K}(r )} dxds\leq
\f{1}{(c_4 r^{-2} m(r)^{\frac{-\tau}{q}})^{2}}  \iint _{Q(x_{i},t_{i};\,r)\cap \tilde{K}(r )}|\nabla u| ^{2} dxds.$$
For $q<2$, by the H\"older inequality and the last inequality, we have
\be\ba\label{wu1}
\iint _{Q(x_{i},t_{i};\,r)\cap \tilde{K}(r )} |\nabla u|^q dxds &\leq\B(\iint _{Q(x_{i},t_{i};\,r)\cap \tilde{K}(r )} |\nabla u|^2dxds\B)^{q/2}  \B(\iint _{Q(x_{i},t_{i};\,r)\cap \tilde{K}(r )} dxds\B)^{(2-q)/2}\\
&\leq Cr^{4-2q} m(r)^{\frac{ \tau(2-q)}{q}}\B(\iint _{Q(x_{i},t_{i};\,r)\cap \tilde{K}(r )} |\nabla u|^2dxds\B).
\ea\ee
By virtue of the definition of $\Psi_\delta(E, h)$ and the above inequality,  we derive from \eqref{order}, \eqref{wu1} and  \eqref{wu2} that, for every $k\geq2$, $0 <r \leq r_{0} \leq 1/2,$
\be\ba
\Psi_{5r}(G_k, t\Gamma^{\sigma}(t)) \leq& \sum_{i}(5r)\Gamma^{\sigma}(5r)\\
\leq& \sum_{i} r  m(r)^{1+\tau} \widetilde{J}_q(r)\\
\leq& C
\Gamma(r)^{(1+\tau)\sigma+\frac{ \sigma\tau(2-q)}{q}}\sum_{i}   \B(\iint _{Q(x_{i},t_{i};\,r)\cap \tilde{K}(r )} |\nabla u|^2dxds\B)
\\
\leq& C \Gamma(r)^{\delta_{2}-1}\iint _{L^{k}(r_0)} \Gamma(d)|\nabla u|^2 dxds,
\ea\label{wu3}
\end{equation}
which together with \eqref{ref}  implies that
\begin{equation}\label{MSF}
\Lambda(\mathcal{S}\cap F(m), r\Gamma(r)^{\sigma })=0.
\end{equation}
To complete the proof, we have to show that $\Lambda(\mathcal{S}\backslash F(m), r\Gamma(r)^{\sigma })=0$. Indeed,   for $(x,t)\in \mathcal{S}\backslash F(m)$, we deduce    from \eqref{ckn} and the definition of $ F(m)$ that
$$\limsup_{r\rightarrow 0} r^{-1} \iint _{Q (x,t;\,r)} |\nabla u|^2 dxds \geq \varepsilon$$
and
\begin{equation*}
 \limsup_{r\rightarrow 0} \f{1} {r\Gamma(r)^\sigma} \iint _{Q (x,t;\,r)} |\nabla u|^2 dxds \geq 1.
\end{equation*}
Let $\delta>0$, for each $(x,t)\in \mathcal{S}\backslash F(m)$, we can choose $Q(x,t;\,r)$ with $r<\delta$ such that
\begin{equation*}
  \iint _{Q (x,t;\,r)} |\nabla u|^2 dxds \geq \varepsilon r/2 \quad \text{and} \quad \iint _{Q (x,t;\,r)} |\nabla u|^2 dxds \geq r\Gamma(r)^\sigma/2.
\end{equation*}
From classical  Vitali covering lemma, we know that there exists a disjoint subfamily $\{Q(x_i,t_{i};{r_i})\}$
such that
$$\mathcal{S}\backslash F(m)\subset \bigcup_i Q(x_i,t_i;\,{5 r_i})$$
and
$$\B|\bigcup_i Q(x_i,t_i;\,{ r_i})\B|\leq C \delta^4 \sum_i r_i \leq C \delta^4   \iint _{\bigcup\limits_i Q(x_i,t_i;\,{ r_i})} |\nabla u|^2 dxds \leq C \delta^4,$$
where $C$ is independent  of  $\delta$. In addition, we know that
$$\sum r_i \Gamma(r_i)^\sigma \leq 2\sum_i \iint _{Q (x_i,t_i;\,{r_i} )} |\nabla u|^2 dxds=2\iint _{\bigcup\limits_{i} Q(x_i,t_i;\,{ r_i})} |\nabla u|^2 dxds.$$
Note that $\delta$ is arbitrary. Therefore, it follows from absolutely continuity of the integral of  $|\nabla u|^2$ that
\begin{equation*}
\Lambda(\mathcal{S}\backslash F(m), r\Gamma(r)^{\sigma })=0,
\end{equation*}
Combined this and \eqref{MSF} implies $\Lambda(\mathcal{S}, r\Gamma(r)^{\sigma })=0$. This   ends the proof of Theorem \ref{main}.
\end{proof}

\section*{Acknowledgement}
 The authors would like to express their deepest gratitude to    two anonymous kind referees  and   the editors
    for  careful reading of our manuscript, the invaluable comments and suggestions which helped to improve the paper greatly. In particular, the proof of \eqref{wu3} was generously suggested by the
referee.
Wang was partially supported by  the National Natural
Science Foundation of China under grant No. 11601492.
Wu was partially supported by  the National Natural
Science Foundation of China under grant No. 11771423 and No.
11671378.

\end{document}